\documentclass[11pt]{amsart}
\usepackage{amsmath,amssymb,graphicx,setspace,verbatim}
\usepackage{comment}
\usepackage{appendix}
\theoremstyle{comment}
\usepackage[latin1]{inputenc}

\input xy
\xyoption{all}
\numberwithin{equation}{section}

\newtheorem{theorem}{Theorem}[section]
\newtheorem{problem}{Problem}[section]
\newtheorem{construction}{Construction}[section]
\newtheorem{lemma}[theorem]{Lemma}

\newtheorem{proposition}[theorem]{Proposition}
\newtheorem{corollary}[theorem]{Corollary}

\newtheorem{example}[theorem]{Example}

\begin{document}

\title{Highly symmetric hypertopes}

\author{Maria Elisa Fernandes}
\address{
Maria Elisa Fernandes, Center for Research and Development in Mathematics and Applications, Department of Mathematics, University of Aveiro, Portugal
}
\email{maria.elisa@ua.pt}

\author{Dimitri Leemans}
\address{Dimitri Leemans, Department of Mathematics, University of Auckland, Private Bag 92019, Auckland 1142, New Zealand
}
\email{d.leemans@auckland.ac.nz}
\author{Asia Ivi\'c Weiss}
\address{Asia Ivi\'c Weiss, Department of Mathematics and Statistics, York University, Toronto, Ontario M3J 1P3, Canada
}
\email{weiss@mathstat.yorku.ca}

\begin{abstract}
We study incidence geometries that are thin and residually connected. 
These geometries generalise abstract polytopes.
In this generalised setting, guided by the ideas from the polytopes theory, we introduce the concept of chirality, a property of orderly asymmetry occurring frequently in nature as a natural phenomenon. 
The main result in this paper is that automorphism groups of regular and chiral thin residually connected geometries need to be $C$-groups in the regular case and $C^+$-groups in the chiral case.
\end{abstract}
\maketitle

\noindent \textbf{2000 Math Subj. Class:} 51E24, 52B11, 20F05.

\section{Introduction}
Guided by the ideas of chirality in the polytope theory (see~\cite{SW1991}), the present paper extends the concept to a more general setting of incidence geometries. Indeed, when an incidence geometry is thin, it is possible to define chirality, as in the case of polytopes. It is then interesting to study how the group-theoretic counterpart of chiral polytopes extends in this more general framework, as a chiral polytope is a thin incidence geometry with a linear diagram.
The purpose of this paper is to explore this more general framework for chiral geometries and also to take this opportunity to look at the regular case in a more detailed way.

In Section~\ref{section2}, we state the basic definitions about regular and chiral hypertopes and generalise a result of McMullen and Schulte on strong connectivity in polytopes (see~\cite[Proposition 2A1]{ARP}) to thin strongly chamber connected incidence geometries that we call hypertopes and that are a natural generalisation of polytopes as it will appear in later sections.
In Section~\ref{section3}, we recall how to construct geometries from groups (that is coset geometries) using an algorithm due to Jacques Tits and how to check residual connectedness and flag-transitivity on coset geometries.
In Section~\ref{section4}, we prove that the automorphism group of a regular hypertope is a smooth quotient of a Coxeter group and we show that the converse is not always true. The main results of this section are found in the following theorems: Theorem~\ref{cgroup} shows that a natural smooth quotient of a Coxeter group can be associated to a regular hypertope; Theorem~\ref{theorem46}  shows that if a geometry $\Gamma$ is constructed from a group $G$ and a generating set of involutions satisfying an intersection property, $\Gamma$ will be a regular hypertope provided $G$ is flag-transitive on $\Gamma$.
In Section~\ref{section5}, we give examples of regular hypertopes of rank 3 that are not polytopes, that is, that do not have a linear diagram.
In Sections~\ref{section6},~\ref{section7} and~\ref{section8}, we characterise the automorphism groups of regular and chiral hypertopes. 
The main results are Theorems~\ref{cplusgroup} and~\ref{elisa} that 
describe the rotational subgroup of the automorphism group of a regular or chiral hypertope.
In particular, the new set of generators given for the automorphism group of chiral hypertopes is an independent generating set while the previous characterisation of Schulte and Weiss of automorphism groups of chiral polytopes (see~\cite{SW1991}) did not give such an independent set of generators.
This is extremely useful for instance if we want to bound the rank of chiral hypertopes or polytopes with prescribed automorphism groups in the spirit of~\cite{FL1}.
We also give examples and classifications of certain types of rank four hypertopes where sections are embeddings of maps or hypermaps on the torus.
In Section~\ref{section9}, we conclude the paper with some open problems and acknowledgements.
\section{Regular and chiral hypertopes}\label{section2} 
In~\cite{Tits61}, Jacques Tits introduced the concept of geometry as an object generalizing the notion of incidence and established its close relation with groups  (see also~\cite[Chapter 3]{Buek95}).
Following~\cite{BuekCohen}, we begin by defining an incidence system $\Gamma$ (also called pre geometry or incidence structure in~\cite{Buek95, Pasini94}).

An {\it incidence system} $\Gamma := (X, *, t, I)$ is a 4-tuple such that
\begin{itemize}
\item $X$ is a set whose elements are called the {\it elements} of $\Gamma$;
\item $I$ is a set whose elements are called the {\it types} of $\Gamma$;
\item $t:X\rightarrow I$ is a {\it type function}, associating to each element $x\in X$ of $\Gamma$ a type $t(x)\in I$;
\item $*$ is a binary relation on $X$ called {\em incidence}, that is reflexive, symmetric and such that for all $x,y\in X$, if $x*y$ and $t(x) = t(y)$ then $x=y$.
\end{itemize}
The {\it incidence graph} of $\Gamma$ is the graph whose vertex set is $X$ and where two vertices are joined provided the corresponding elements of $\Gamma$ are incident.
A {\it flag} is a set of pairwise incident elements of $\Gamma$, i.e. a clique of its incidence graph.
The {\it type} of a flag $F$ is $\{t(x) : x \in F\}$.
A {\it chamber} is a flag of type $I$.
An element $x$ is {\em incident} to a flag $F$ and we write $x*F$ for that, provided $x$ is incident to all elements of $F$.
An incidence system $\Gamma$ is a {\it geometry} or {\it incidence geometry} provided that every flag of $\Gamma$ is contained in a chamber (or in other words, every maximal clique of the incidence graph is a chamber).
The {\it rank} of $\Gamma$ is the number of types of $\Gamma$, namely the cardinality of $I$.

We now define the notion of residue which is central in incidence geometry.
Let $\Gamma:= (X, *, t, I)$ be an {\it incidence geometry} and $F$ a flag of $\Gamma$.
The {\em residue} of $F$ in $\Gamma$ is the incidence geometry $\Gamma_F := (X_F, *_F, t_F, I_F)$ where
\begin{itemize}
\item $X_F := \{ x \in X : x * F, x \not\in F\}$;
\item $I_F := I \setminus t(F)$;
\item $t_F$ and $*_F$ are the restrictions of $t$ and $*$ to $X_F$ and $I_F$.
\end{itemize}

An incidence system $\Gamma$ is {\em residually connected} when each residue of rank at least two of $\Gamma$ has a connected incidence graph. It is called {\it thin} (resp. {\it firm}) when every residue of rank one of $\Gamma$ contains exactly two (resp. at least two) elements.

An incidence system $\Gamma:= (X, *, t, I)$ is {\em chamber-connected} when for each pair of chambers $C$ and $C'$, there exists a sequence of successive chambers $C =: C_0,$ $C_1, \ldots, C_n := C'$ such that $\mid C_i\cap C_{i+1} \mid = \mid I\mid -1$.
An incidence system $\Gamma:= (X, *, t, I)$ is {\em strongly chamber-connected} when all its residues of rank at least 2 are chamber-connected.

We now state the following proposition which is a generalisation of Proposition 2A1 of~\cite{ARP}.
\begin{proposition}\label{rcsc}
Let $\Gamma$ be a firm incidence geometry. Then $\Gamma$ is residually connected if and only if $\Gamma$ is strongly chamber-connected.
\end{proposition}
\begin{proof}
Observe that it is enough to prove that $\Gamma$ has a  connected incidence graph if and only if $\Gamma$ is chamber-connected, as this can then be applied to every residue.

Suppose first that the incidence graph of $\Gamma$ is connected.
As in~\cite{ARP}, the proof proceeds by induction on $n := rank(\Gamma)$.
Let $n=2$. Given a pair of chambers  $C_1$ and  $C_k$ there is a sequence of incident elements, having alternate types, whose first two elements are the elements of $C_1$ and the last two are the elements of $C_2$. Pairs of consecutive elements in this sequence give a sequence of chambers needed to give chamber connectedness.
Let $n> 2$. Consider two chambers, $C$ and $C'$, of $\Gamma$.

If $C\cap C' \neq \emptyset$, say $C\cap C' = \{x_1, \ldots, x_k\}$ with $k\geq 1$, then the residue $\Gamma_{x_1}$ contains $C\setminus\{x_1\}$ and $C'\setminus\{x_1\}$ as chambers. Moreover,  $\Gamma_{x_1}$ has a connected incidence graph,  otherwise there is a pair of elements incident to $x_1$ in different components of the incidence graph of $\Gamma$, implying that the geometry in not firm, a contradiction. Hence, by induction, we can find a sequence of successive chambers $C\setminus\{x_1\} =: C_0, C_1, \ldots, C_n := C'\setminus\{x_1\}$ in $\Gamma_{x_1}$ such that $\mid C_i\cap C_{i+1} \mid = n -2$.
The sequence $C =:C_0\cup\{x_1\}, C_1\cup\{x_1\}, \ldots, C_n\cup\{x_1\} := C'$ is then such that $\mid (C_i\cup \{x_1\}) \cap (C_{i+1}\cup \{x_1\}) \mid = n-1$ as needed.

Let $C\cap C' = \emptyset$.
Since $\Gamma$ has  a connected incidence graph, we can find a sequence of successively incident elements $x_0, \ldots, x_k$ such that $x_0 \in C$ and $x_k \in C'$.
For each $i = 1, \ldots, k$, there is a chamber $C_i \supseteq \{x_{i-1}, x_i\}$. Set $C_0 := C$ and $C_k := C'$.
We now appeal to the first part of the proof to move $C_{i-1}$ to $C_i$ by a sequence of adjacent chambers containing $x_i$ for each $i=1,\ldots, k$; concatenation then gives the required sequence from $C$ to $C'$.

Suppose now that $\Gamma$ is chamber-connected.
Connectedness is obvious as every flag is contained in a chamber. Indeed, let $x_1$ and $x_2$ be two elements of $\Gamma$. Each of them is contained in at least one chamber. Let $x_1\in C_1$ and $x_2\in C_2$ where $C_1$ and $C_2$ are chambers of $\Gamma$, we easily get a path from $x_1$ to $x_2$ using the sequence of successive chambers connecting $C_1$ to $C_2$.
Hence $\Gamma$ has a connected incidence graph.
\end{proof}

A {\em weak hypertope} is a thin incidence geometry.
A {\em hypertope} is a weak hypertope which is strongly chamber connected or equivalently residually connected.

Let $\Gamma:=(X,*, t,I)$ be an incidence system.
An {\em automorphism} of $\Gamma$ is a mapping
$\alpha:(X,I)\rightarrow (X,I):(x,t(x)) \mapsto (\alpha(x),t(\alpha(x)))$
where
\begin{itemize}
\item $\alpha$ is a bijection on $X$;
\item for each $x$, $y\in X$, $x*y$ if and only if $\alpha(x)*\alpha(y)$;
\item for each $x$, $y\in X$, $t(x)=t(y)$ if and only if $t(\alpha(x))=t(\alpha(y))$.
\end{itemize}
Note that $\alpha$ induces a bijection on $I$.

An automorphism $\alpha$ of $\Gamma$ is called {\it type preserving} when for each $x\in X$, $t(\alpha(x))=t(x)$ (i.e. $\alpha$ maps each element on an element of the same type).

The set of type-preserving automorphisms of $\Gamma$ is a group denoted by $Aut_I(\Gamma)$.
The set of automorphisms of $\Gamma$ is a group denoted by $Aut(\Gamma)$.
A {\em correlation} is a non-type-preserving automorphism, that is an element of $Aut(\Gamma)\setminus Aut_I(\Gamma)$.

An incidence geometry $\Gamma$ is {\em flag-transitive} if $Aut_I(\Gamma)$ is transitive on all flags of a given type $J$ for each type $J \subseteq I$.
An incidence geometry $\Gamma$ is {\em chamber-transitive} if $Aut_I(\Gamma)$ is transitive on all chambers of $\Gamma$.
Finally, an incidence geometry $\Gamma$ is {\em regular} if $Aut_I(\Gamma)$ acts regularly on the chambers (i.e. the action is semi-regular and transitive).
The following proposition is folklore in incidence geometry. We recall it here as our paper makes bridges between different subjects.
\begin{proposition}
Let $\Gamma$ be an incidence geometry. $\Gamma$ is chamber-transitive if and only if $\Gamma$ is flag-transitive.
\end{proposition}
\begin{proof}
It is obvious that if $\Gamma$ is flag-transitive, then $\Gamma$ is chamber-transitive.
Suppose $\Gamma$ is chamber-transitive.
Let $F_1$ and $F_2$ be two flags of the same type. Each of them is contained in at least one chamber as $\Gamma$ is a geometry. Let $C_i$ be a chamber containing $F_i$ ($i=1,2$).
Since $\Gamma$ is chamber-transitive, there exists an element $g\in Aut_I(\Gamma)$ such that $g(C_1) = C_2$. In particular, as $g$ preserves the types of the elements of $\Gamma$, we have $g(F_1) = F_2$.
\end{proof}

A {\em regular weak hypertope} is a flag-transitive weak hypertope.
A {\em regular hypertope} is a flag-transitive hypertope.
We use the adjective ``regular" here as, if $\Gamma$ is a weak hypertope, it is thin and hence the action of $Aut_I(\Gamma)$ is necessarily free.
Indeed, if  $g\in Aut_I(\Gamma)$ fixes a chamber $C$ then, as there is exactly one chamber $C^i$ differing from $C$ in the $i$-element, for each $i\in I$,  $g$ also fixes $C^i$.  As the geometry is residually connected $g$ must be the identity.

We can also extend the notion of chirality in abstract polytopes to the more general framework of incidence geometries.
Although thinness is not necessary in order to define what is a regular geometry, it is needed to define chiral geometries.

Two chambers $C$ and $C'$ of an incidence geometry of rank $r$ are called {\it $i$-adjacent} if $C$ and $C'$ differ only in their $i$-elements. We then denote $C'$ by $C^i$.
Let $\Gamma(X,*,t,I)$ be a thin incidence geometry.
We say that $\Gamma$ is {\it chiral} if
$Aut_I(\Gamma)$ has two orbits on the chambers of $\Gamma$ such that any two adjacent chambers lie in distinct orbits.
Moreover, if $\Gamma$ is residually connected, we call $\Gamma$ a {\it chiral hypertope}.

When $\Gamma$ is a chiral hypertope, if $Aut(\Gamma) \neq Aut_I(\Gamma)$, correlations may either interchange the two orbits or preserve them. A correlation that interchanges the two orbits is said to be {\em improper} and a correlation that preserves them is said to be {\em proper}. 
Correlations, in the case of polytopes, are called dualities in~\cite{HW2005}.
In that paper, it is shown that if a chiral polytope has a duality of one kind, all its dualities will be of the same kind (see~\cite[Lemma 3.1]{HW2005}). This result does not extend to chiral hypertopes. 

A rank one hypertope is a geometry with two elements. The polygons are precisely the hypertopes of rank two. Abstract regular polytopes are regular hypertopes. More details about this correspondence will be given in Section~\ref{section5}. In rank three and higher there are (regular) hypertopes that are not abstract (regular) polytopes. More generally there are examples of regular geometries that are not thin (see for instance geometry number 2 of $Sym(3)$ in~\cite{BDL96b}).
\section{Regular hypertopes as coset geometries}\label{section3}

Given an incidence system $\Gamma$ and a chamber $C$ of $\Gamma$, we may associate to the pair $(\Gamma,C)$ a pair consisting of a group $G$ and a set $\{G_i : i \in I\}$ of subgroups of $G$ where $G := Aut_I(\Gamma)$ and $G_i$ is the stabilizer in $G$ of the element of type $i$ in $C$.
The following proposition shows how to reverse this construction, that is starting from a group and some of its subgroups, construct an incidence system.
\begin{proposition}(Tits, 1956)~\cite{Tits61}\label{tits}
Let $n$ be a positive integer
and $I:= \{1,\ldots ,n\}$.
Let $G$ be a group together with a family of subgroups ($G_i$)$_{i \in I}$, $X$ the set consisting of all cosets $G_ig$ with $g \in G$ and $i \in I$, and $t : X \rightarrow I$ defined by $t(G_ig) = i$.
Define an incidence relation $*$ on $X\times X$ by :
\begin{center}
$G_ig_1 * G_jg_2$ iff $G_ig_1 \cap G_jg_2 \neq \emptyset$.
\end{center}
Then the 4-tuple $\Gamma := (X, *, t, I)$ is an incidence system having a chamber.
Moreover, the group $G$ acts by right multiplication as an automorphism group on $\Gamma$.
Finally, the group $G$ is transitive on the flags of rank less than 3.
\end{proposition}
In particular, we want to find under which conditions the construction above produces incidence geometries that are regular hypertopes.

Observe that, in the proposition above, $G \leq Aut_I(\Gamma)$.
When a geometry $\Gamma$ is constructed using the proposition above, we denote it by $\Gamma(G; (G_i)_{i\in I})$ and call it a {\em coset geometry}.
The subgroups $(G_i)_{i\in I}$ are called the {\it maximal parabolic subgroups}.
The {\it Borel subgroup} of the incidence system is the subgroup $B = \cap_{i \in I} G_i$.
The action of $G$ on $\Gamma$ involves a {\it kernel} $K$ which is the largest normal subgroup of $G$ contained in every $G_i$, $i \in I$. 
The kernel is the identity if and only if  $G$ acts faithfully on $\Gamma$.
If $G$ acts transitively on all chambers of $\Gamma$, hence also on
all flags of any type $J$, where $J$ is a subset of $I$, we say that $G$ is {\em flag-transitive} on $\Gamma$
or that $\Gamma$ is {\em flag-transitive (under the action of $G$)}.
In that case, any chamber of $\Gamma$ is obtained by multiplying the cosets of the {\em base chamber} $\{G_0, \ldots, G_{r-1}\}$ on the right by an element $g\in G$.
If $G$ acts {\it regularly} on $\Gamma$ (i.e. 
the action is free and flag-transitive) we say that $\Gamma$ is {\em regular} (under the action of $G$).

If $\Gamma=\Gamma(G; (G_i)_{i\in I})$ is a flag-transitive geometry and $F$ is a flag of $\Gamma$,
the {\it residue} of $F$ is isomorphic to the incidence system
\begin{center}
$\Gamma_F = \Gamma(\cap_{j \in t(F)} G_j,(G_i \cap(\cap_{j \in t(F)}G_j))_{i \in I \backslash t(F)})$
\end{center}
and we readily see that $\Gamma_F$ is also a flag-transitive geometry.\\

If $\Gamma$ is a geometry of rank $2$ with $I = \{i, j\}$ such that each of its $i$-elements
is incident with each of its $j$-elements, then we call $\Gamma$ a {\it generalized digon}.

We refer to~\cite{Buek79} and \cite{Buek83} for the notion of the Buekenhout diagram of a geometry and simplify it here to deal only with thin geometries.
For a thin, residually connected, flag-transitive
coset geometry $\Gamma(G; (G_i)_{i\in I})$, the {\em Buekenhout diagram} $\mathcal{B}(\Gamma)$ is a graph whose vertices
are the elements of $I$.
Elements $i$, $j$ of $I$ are not joined by an edge of the diagram provided that a residue $\Gamma_F$
of type $\{i, j\}$ is a generalized digon.
Otherwise, $i$ and $j$ are joined by an edge. This edge is endowed with a number $g_{ij}$ that is equal to
half the girth of the incidence graph of a residue $\Gamma_F$ of type $\{ i, j \}$ provided that $g_{ij}>3$.

For example consider the cube as an incidence geometry of rank three with elements being the vertices, edges and faces, type set being $\{0,1,2\}$, vertices (resp. edges, faces) being of type 0 (resp. 1, 2) and incidence being symmetrised inclusion. The vertices of the Buekenhout diagram correspond to the types $0,\, 1$ and $2$. The residue of an edge, an element of type 1,  is a generalized digon,  the residue of a vertex, an element of type 0, is a triangle and the residues of a face is a square. Thus $g_{12}=3$, $g_{01}=4$ and the diagram is linear as in the following figure.

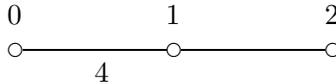
\begin{center}
\begin{picture}(120,40)
\put(0,15){\circle{5}}
\put(60,15){\circle{5}}
\put(120,15){\circle{5}}
\put(30,3){4}
\put(-3,25){0}
\put(57,25){1}
\put(117,25){2}
\put(3,15){\line(1,0){54}}
\put(63,15){\line(1,0){54}}
\end{picture}
\end{center}

In the case of regular polytopes the Buekenhout diagram and the Coxeter diagram are essentially the same.

In order for a coset geometry $\Gamma = \Gamma(G;(G_i)_{i\in I})$ to be a regular hypertope, it needs in particular to be residually connected.
The following Theorem translates the residual connectedness (or strong flag-connectedness) condition into a group-theoretic condition.
\begin{theorem}\cite{Dehon94}\label{RCGroupal}
Let $\Gamma = \Gamma(G;(G_i)_{i\in I})$ be a flag-transitive coset geometry.
$\Gamma$ is residually connected if and only if for every subset $J\subseteq I$ of cardinality at most $|I|-2$,
\[\cap_{j\in J} G_j = \langle \cap_{j \in J\cup \{k\}} G_j : k \in I-J\rangle.\]
\end{theorem}

The following theorem, due to Jacques Tits, shows that if a coset geometry $\Gamma(G;(G_i)_{i\in I})$ has a non-trivial kernel $K$ then this geometry can be constructed from a smaller group, namely $G/K$. 

\begin{theorem}\cite{Tits61}
Let $\Gamma(G;(G_i)_{i\in I})$ be a coset geometry. If $K$ is the kernel of the action of $G$ on $\Gamma$, then  $\Gamma(G;(G_i)_{i\in I}) \cong \Gamma(G/K;(G_i/K)_{i\in I})$.
\end{theorem}
For this reason it is natural to assume that a thin flag-transitive coset geometry is regular.
The following lemma shows that we then may assume the Borel subgroup of a hypertope $\Gamma(G;(G_{i})_{i\in I})$ to be the identity subgroup.

\begin{lemma}\label{Borel}
Let $\Gamma(G;(G_{i})_{i\in I})$ be a regular hypertope.
Then $B = \cap_{i\in I}G_i = 1$.
\end{lemma}
{\bf Proof:}
The groups  $G_{\overline{J}}:=\cap_{j\in J}G_j$ where $J$ is any subset of $I$ of cardinality $\mid I\mid - 1$ contain $B$ as a subgroup of index 2, for otherwise, thinness is contradicted. Hence $B$ is a normal subgroup of all these groups $G_{\overline{J}}$. Now, as $\Gamma$ is regular and residually connected, by Theorem~\ref{RCGroupal}, the subgroups  $G_{\overline{J}}$ generate $G$ and thus $B$ must also be a normal subgroup of $G$. This means $B$ is a kernel. Then in order to have a free action of $G$ on $\Gamma$, we must have $B = 1$.
\hfill $\Box$

Observe that the subgroups $G_{\overline{J}}$ appearing in the proof above  are cyclic groups of order 2. This will enable us to make the connection between thin regular coset geometries and groups generated by involutions.

For $\Gamma(G;(G_{i})_{i\in I})$ a regular hypertope, we define $\rho_i$ as the generator of the minimal parabolic subgroup $\cap_{j\in I\setminus\{i\}} G_j$. Observe that all the $\rho_i$'s are involutions.
We call the set $\{\rho_i : i \in I\}$  the {\it distinguished generators} of $\Gamma(G;(G_{i})_{i\in I})$.

The following result gives a way to check whether a coset geometry (and in particular a hypertope) is flag-transitive. See also Dehon \cite{Dehon94}.

\begin{theorem} {\emph{(Buekenhout, Hermand \cite{BH91})}}\label{hermandft}
 Let $\mathcal{P}(I)$ be the set of all the subsets of $I$ and let $\alpha: \mathcal{P}(I)\setminus \{ \emptyset \} \to I$ be a function such that $\alpha (J) \in J$ for every $J \subset I$, $J \neq \emptyset$. The coset geometry $\Gamma = \Gamma(G;(G_{i})_{i\in I})$ is flag-transitive if and only if, for every $J \subset I$ such that $\vert J \vert \geq 3$, we have \[ \bigcap_{j \in J - \alpha (J)} (G_j G_{\alpha(J)}) = \Bigg( \bigcap_{j \in J - \alpha (J)} G_j \Bigg) G_{\alpha(J)}. \]
\end{theorem}

A proof of this result is also available in the book by Buekenhout and Cohen (see~\cite[Theorem 1.8.10]{BuekCohen}).

\section{C-groups}\label{section4}
A {\em C-group of rank $r$} is a pair $(G,S)$ such that $G$ is a group and $S:=\{\rho_0,\ldots,\rho_{r-1}\}$ is a generating set of involutions of $G$  that satisfy the following property.
\begin{equation}\label{IC}
\forall I, J \subseteq \{0, \ldots, r-1\},
\langle \rho_i \mid i \in I\rangle \cap \langle \rho_j \mid j \in J\rangle = \langle \rho_k \mid k \in I \cap J\rangle.
\end{equation}

\noindent This property is called the {\em intersection property} and denoted by $IP$. We call any subgroup of $G$ generated by a subset of $S$ a {\em parabolic subgroup} of the $C$-group $(G,S)$. In particular, a parabolic subgroup generated by exactly one involution of $S$ is called {\em minimal} and a parabolic subgroup generated by all but one involutions of $S$ is called {\em maximal}. We write $G_J:=\langle \rho_j\mid j \in J\rangle$ for $J\subseteq \{0, \ldots, r-1\}$ and $G_i:=G_{I\setminus\{i\}}$. Obviously, $G_\emptyset = \{1_G\}$ and $G_{\{0, \ldots, r-1\}} = G$. 

A C-group is a {\em string $C$-group} provided its generating involutions can be reordered in such a way that $(\rho_i\rho_j)^2 = 1_G$ for all $i,j$ with $|i-j| > 1$.

We say that two C-groups $(G,S)$ and $(G',S')$ are {\em isomorphic} if there is an isomorphism $\alpha:G\rightarrow G'$ such that $\alpha(S) = S'$.

The {\em Coxeter diagram} $\mathcal{C}(G,S)$ of a C-group $(G,S)$ is a graph whose vertex set is $S$. Two vertices $\rho_i$ and $\rho_j$ are joined by an edge labelled by $o(\rho_i\rho_j)$. As a consequence, the Coxeter diagram is a complete graph. We take the convention of dropping an edge if its label is 2 and of not writing the label if it is 3.
The Coxeter diagram of a string C-group has a string shape.

From now on, we will construct hypertopes as coset geometries.



\begin{theorem}\label{cgroup}
Let $I:=\{0, \ldots, r-1\}$ and let $\Gamma:= \Gamma(G;(G_i)_{i\in I})$ be a regular hypertope of rank $r$.
 The pair $(G,S)$ where $S$ is the set of distinguished generators of $\Gamma$ is a C-group of rank $r$.
\end{theorem}
\begin{proof}
The theorem is obviously true for $r = 2$ as the maximal parabolic subgroups are cyclic and have trivial intersection.
Suppose the theorem is true for $r-1$ and let us show it is then true for $r$ by way of contradiction.
Let us denote by $G_{\overline{K}}$ the subgroup $\langle \rho_k \mid k \in I \setminus K\rangle$.


If $(G,S)$ does not satisfy ($\ref{IC}$), then there is a pair of subgroups, $G_{\overline{K}}$ and $G_{\overline{J}}$ with $K,J \subseteq I$ such that $G_{\overline{K}} \cap G_{\overline{J}} \neq G_{\overline{K \cup J}}$.
Hence $G_{\overline{K}} \cap G_{\overline{J}} > G_{\overline{K \cup J}}$. Take $g \in (G_{\overline{K}} \cap G_{\overline{J}}) \setminus G_{\overline{K \cup J}}$.
This $g$ fixes a flag of type $K \cup J$ in the base chamber $\{G_0, \ldots, G_{r-1}\}$.
But the action of $G_{\overline{K\cup J}}$ must be regular on the residue $\Gamma_F$ of the flag $F:=\{G_i \mid i \in K \cup J\}$. Indeed, that residue is also a thin regular residually connected geometry and its distinguished generators are exactly those of $G_{\overline{K\cup J}}$ and satisfy (\ref{IC}) by induction.
Any element of $G_{\overline{K\cup J}}$ will fix all elements of $\{G_j \mid j \in K \cup J\}$.
Since $G_{\overline{K\cup J}}$ is regular on $\Gamma_F$, there must exist an element $h \in G_{\overline{K\cup J}}$ that sends the flag $\{G_k \mid k \in I\setminus (K\cup J)\}$ onto $\{G_k*g \mid k \in I\setminus (K\cup J)\}$.
But then $g*h^{-1}\not = 1_G$ fixes the base chamber $\{G_i \mid i \in I\}$, a contradiction with the regularity of the action of $G$ on the chambers of $\Gamma$.
\end{proof}



Observe that we can construct a coset geometry $\Gamma(G;(G_i)_{i\in I})$ in a natural way from a C-group $(G,S)$ of rank $r$ by letting $G_i=G_{I\setminus\{i\}}$ as above.
This construction always gives  a thin, residually connected, regular coset geometry when the rank is at most 2.
We next show that this construction will not always give, for rank three (and therefore also not for higher ranks), a thin, residually connected, regular coset geometry. To this end, we recall a group-theoretical result of Tits.

\begin{lemma} (Tits~\cite{Tits74})\label{titsft}
Let $G_0,G_1,G_2$ be three subgroups of a group $G$. Then the following conditions are equivalent.
 \begin{enumerate}
  \item $G_0G_1 \cap G_0G_2 = G_0(G_1 \cap G_2)$
  \item $(G_0 \cap G_1) \cdot (G_0 \cap G_2) = (G_1G_2) \cap G_0$
  \item If the three cosets $G_0x$, $G_1y$ and $G_2z$ have pairwise nonempty intersection, then $G_0x \cap G_1y \cap G_2z \neq \emptyset$.
 \end{enumerate}
\end{lemma}

\begin{proposition}\label{cgrouprk3}
Let $(G,\{\rho_0, \rho_1, \rho_2\})$ be a C-group of rank three
 and let $\Gamma:=\Gamma(G;\{\langle \rho_1, \rho_2\rangle,\langle \rho_0, \rho_2\rangle,\langle \rho_0, \rho_1\rangle\})$.
Then $\Gamma$ is thin if and only if $G$ is regular on $\Gamma$. Moreover, if $\Gamma$ is thin (or regular), it is a regular hypertope.
\end{proposition}
\begin{proof}
First as $G_0\cap G_1\cap G_2=\{1_G\}$, $G$ acts freely on $\Gamma$.
Now suppose that $G$ is transitive on the chambers of $\Gamma$.
Take a flag of type $\{i,j\}$. Without loss of generality, we can assume this flag to be $\{G_i,G_j\}$ by the last sentence of Proposition~\ref{tits}. As $G_i \cap G_j  \cong C_2$ there are exactly two elements of type $k$ incident to $G_i$ and $G_j$. Hence $\Gamma$ is thin.

Conversely,  suppose that $\Gamma$ is thin.
If $\Gamma$ is not flag transitive, by Lemma~\ref{titsft} and Theorem~\ref{hermandft}, there exists a chamber $\{G_0x,G_1y,G_2z\}$ such that $G_0x\cap G_1y\cap G_2z=\emptyset$. Let $h\in G_0x\cap G_1y$. Then  $G_2h$ and $G_2\rho_2h$ as well as $G_2z$ are incident to both $G_0x$ and $G_1y$. Hence the residue of the flag $\{G_0x, G_1y\}$ has at least three elements and $\Gamma$ is not thin.
When $\Gamma$ is flag transitive we use Theorem~\ref{RCGroupal} to conclude that $\Gamma$ is residually connected. By Proposition~\ref{rcsc}, $\Gamma$ is therefore a regular hypertope.
\end{proof}

The following example shows that C-groups can give rise to geometries that are neither thin nor regular.
\begin{example}\label{ex44}
The toroidal hypermap usually denoted by $(3,3,3)_{(1,1)}$, can be constructed from the group $G \cong E_9:C_2$ of order 18.
Defining relations for the automorphism group are
\[ \rho_0^2 = \rho_1^2 = \rho_2^2 = (\rho_0 \rho_1 \rho_2)^2 = (\rho_0 \rho_1)^3 = (\rho_1 \rho_2)^3 = 1.\]
The pair $(G,\{\rho_0,\rho_1,\rho_2\})$ is a C-group (satisfying IP). However, the incidence graph of the coset geometry $\Gamma:=\Gamma(G;\{\langle \rho_1, \rho_2\rangle,\langle \rho_0, \rho_2\rangle,\langle \rho_0, \rho_1\rangle\})$ is a $K_{3,3,3}$ and hence $G$ cannot be flag-transitive on $\Gamma$ as $\Gamma$ has $3^3 = 27$ chambers. Moreover, $\Gamma$ is not thin. It is residually connected though.
Observe that, as the incidence graph of $\Gamma$ is a $K_{3,3,3}$, $Aut_I(\Gamma)\cong S_3\times S_3\times S_3$ is flag-transitive on $\Gamma$ but $G$ is not.
\end{example}
The next example shows that even residual connectedness may be lost in higher ranks.
\begin{example}\label{ex4.5}
Let $G \cong A_6$ and define $S := \{\rho_0 = (1,2)(3,4), \rho_1 =(2,6)(3,5), \rho_2=(1,4)(2,3), \rho_3=(1,4)(3,5)\}$.
It can be checked by hand or using {\sc Magma}~\cite{magma} that $(G,S)$ is a C-group.
This C-group was mentioned in~\cite{CJL2013}.
It has the following Coxeter diagram and sublattice.

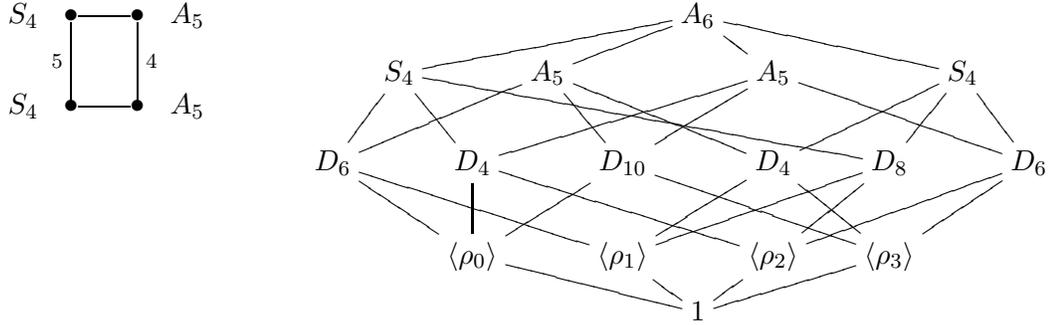
\begin{center}
\begin{tabular}{cc}
$\xymatrix@-1.5pc{
S_4& *{\bullet} \ar@{-}[rr] \ar@{-}[dd]_5&& *{\bullet} \ar@{-}[dd]^4 & A_5 \\
&&&&&&\\
S_4 &*{\bullet}  \ar@{-}[rr] && *{\bullet}&A_5}$
&
$\xymatrix@-1.5pc{
&&&&&A_6\ar@{-}[dll]\ar@{-}[dllll] \ar@{-}[dr] \ar@{-}[drrrr]&&&&&  \\
&S_4\ar@{-}[ddl] \ar@{-}[ddr]\ar@{-}[ddrrrrrrr]&&A_5\ar@{-}[ddlll]\ar@{-}[ddr]\ar@{-}[ddrrr]&&&A_5\ar@{-}[ddllll]\ar@{-}[ddll]\ar@{-}[ddrrrr]&&&S_4\ar@{-}[ddlll] \ar@{-}[ddl]\ar@{-}[ddr]&\\
&& && && && &&\\
 D_6\ar@{-}[ddrr]\ar@{-}[ddrrrr]&&D_4\ar@{-}[dd]\ar@{-}[ddrrrr]&&D_{10}\ar@{-}[ddll]\ar@{-}[ddrrrr]&& D_4\ar@{-}[ddrr]\ar@{-}[ddll]&&D_8\ar@{-}[ddllll]\ar@{-}[ddll]&&D_6\ar@{-}[ddllll]\ar@{-}[ddll]\\
  &&&&&&&&&&\\
 &&\langle \rho_0\rangle\ar@{-}[drrr]&&\langle \rho_1\rangle\ar@{-}[dr]&&\langle \rho_2\rangle\ar@{-}[dl]&&\langle \rho_3\rangle\ar@{-}[dlll]&&\\
 &&&&&1&&&&&}$
 \end{tabular}
\end{center}

However it is not giving a thin residually connected regular geometry with the construction above.
Indeed, the subgroups $G_1 := \langle (1,2)(3,4), (1,4)(2,3), (1,4)(3,5) \rangle$ and $G_2 := \langle (1,2)(3,4),(2,6)(3,5), (1,4)(3,5)\rangle$ are both isomorphic to $A_5$ and their intersection is a dihedral group of order 10. This means that the corresponding coset geometry will have 6 elements of type 1 and 6 elements of type 2, and that each element of type 1 is incident to each element of type 2. But then, the residue of an element of type $\{0, 3\}$, which is supposed to be a 4-gon by the Coxeter diagram, consists of 4 elements of type 1 and four elements of type 2, having a complete bipartite graph $K_{4,4}$ as incidence graph. Therefore, the coset geometry cannot be thin.
In addition, it can be checked with {\sc Magma} that the coset geometry is also neither residually connected nor flag-transitive. 
\end{example}

Given a thin, residually connected, regular coset geometry $\Gamma:= \Gamma(G;(G_i)_{i\in I})$ and the pair $(G,S)$ defined as above, we have $\mathcal{C}(G,S) \cong \mathcal{B}(\Gamma)$.

\begin{theorem}\label{theorem46}
Let $(G,\{\rho_0, \ldots, \rho_{r-1}\})$ be a C-group of rank $r$ and let $\Gamma := \Gamma(G;(G_i)_{i\in I})$ with $G_i := \langle \rho_j | \rho_j \in S, j \in I\setminus \{i\} \rangle$ for all $i\in I:=\{0, \ldots, r-1\}$.
If $G$ is flag-transitive on $\Gamma$, then
$\Gamma$ is a regular hypertope.
\end{theorem}
\begin{proof}
Residual connectedness follows from Theorem~\ref{RCGroupal}, the fact that $\Gamma$ is flag-transitive and the definition of $\Gamma$.
The minimal parabolic subgroups of $\Gamma$ are
cyclic groups of order 2 by the intersection property of the C-group, hence $\Gamma$ is thin.
Therefore $\Gamma$ is a hypertope, it is flag-transitive  and from the intersection property the intersection of all maximal parabolic subgroups of $\Gamma$ must be reduced to the identity, hence $\Gamma$ is regular.
\end{proof}

\section{Polytopes and hypermaps}\label{section5}
An abstract polytope ${\mathcal{P}}$ is a ranked partially ordered set whose elements are called {\it faces}. A polytope $\mathcal{P}$ of rank $n$ has faces of ranks $-1,0,\ldots,n$; $\mathcal{P}$ has a smallest and a largest face, of ranks $-1$ and $n$, respectively. Each maximal chain (or chamber) of $\mathcal{P}$ contains $n+2$ faces, one for each rank.  $\mathcal{P}$ is strongly chamber-connected. $\mathcal{P}$ is {\it thin},  that is, for every flag and every $j = 0,\ldots,n-1$, there is precisely {\it one} other ({\em $j$-adjacent}) flag with the same faces except the $j$-face. This condition is also called the \emph{diamond condition}.

Abstract regular polytopes can be identified with string C-groups as shown in~\cite[Theorem 2E11]{ARP}. In this case, the involutions of a string C-group $(G,S)$ can be ordered in such a way that $(\rho_i\rho_j)^2 = 1_G$ for all $i,j$ with $|i-j| > 1$.  The \emph{Schl\"afli type} of a regular polytope ${\mathcal P}$ is given by $\{p_1,p_2,\ldots, p_{n-1}\}$ where $p_i$ is the order of the product of consecutive generators $\rho_{i-1}\rho_i$.
The right cosets of the maximal parabolic subgroups of $(G,S)$ correspond to faces of the polytope.
The rank of the faces of the polytope is induced by the labeling of the generators of $(G,S)$.
By reversing the order of the generators of a string C-group, one obtains the dual of the corresponding polytope.

The main theorem of~\cite{Aschbacher83} can be rephrased in the framework of string C-groups as follows.

\begin{theorem}[Aschbacher, 1983]~\cite{Aschbacher83}
Let $(G,\{\rho_0, \ldots, \rho_{r-1}\})$ be a string C-group of rank $r$ and let $\Gamma := \Gamma(G;(G_i)_{i\in I})$ with $G_i := \langle \rho_j | \rho_j \in S, j \in I\setminus \{i\} \rangle$ for all $i\in I:=\{0, \ldots, r-1\}$.
Then $\Gamma$ is thin, residually connected and regular. Moreover, $\Gamma$ has a string diagram.
\end{theorem}
\begin{proof}
The intersection property of $(G,S)$ implies assumption (i) of Aschbacher's theorem while the string condition implies assumption (ii). Therefore we can apply his result to string C-groups to show that $\Gamma$ is indeed thin, residually connected and regular in all ranks.
\end{proof}

Therefore, abstract regular polytopes, being also string C-groups, are thin regular residually connected coset geometries (or hypertopes).
The diamond condition in polytopes corresponds to thinness in geometries.
A chain (respectively flag) of a polytope is a flag (respectively chamber) in the corresponding geometry.
Strong flag-connectedness in polytopes corresponds to residual connectedness in geometries. The commuting property of non-consecutive generators in string C-groups corresponds to the linearity of the Buekenhout diagram in geometries.
The concept of adjacent flags in polytopes is equivalent to that of adjacent chambers.
Two chambers of an incidence geometry $\Gamma$ of rank $r$ are {\em adjacent} provided their intersection is a flag of cardinality $r-1$.
\begin{theorem}
Let $\Gamma := \Gamma(G;(G_i)_{i\in I})$ be a thin, residually connected, regular coset geometry with a string diagram.
Let $C:= (G,\{\rho_0, \ldots, \rho_{r-1}\})$ where $\{\rho_0, \ldots, \rho_{r-1}\}$ is the set of distinguished generators of $\Gamma$.
Then $C$ is a string C-group.
\end{theorem}
\begin{proof}
This is an immediate consequence of Theorem~\ref{cgroup}.
\end{proof}

In rank three, a string C-group induces an abstract regular polyhedron with vertices, edges and faces. When finite, such a polyhedron
can be embedded on a closed surface (orientable or not) without boundary and is usually called reflexible map~\cite{CoxeterMoserED4}.
However, not all maps are abstract polytopes as some of them do not satisfy the diamond condition.

Thin regular geometries induced by rank three non-string C-groups, provide examples of reflexible hypermaps. For example, the following Coxeter diagram represents a string C-group $K$ generated by three reflections $\rho_0$, $\rho_1$ and $\rho_2$, that is universal in this case, meaning it is the full Coxeter group $[6,3]$.
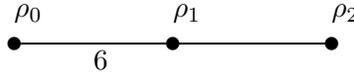
\begin{center}
\begin{picture}(120,25)
\put(0,10){\circle*{5}}
\put(60,10){\circle*{5}}
\put(120,10){\circle*{5}}
\put(30,0){6}
\put(0,10){\line(1,0){120}}
\put(0,20){$\rho_0$}
\put(60,20){$\rho_1$}
\put(120,20){$\rho_2$}
\end{picture}
\end{center}

This C-group is the symmetry group of the regular tessellation of the Euclidean plane by hexagons which is an abstract regular polyhedron of Schl\"afli type $\{6,3\}$.
The regular maps with the same Schl\"afli type $\{6,3\}$ will be embedded on the torus
and induced by adding the relation
\begin{equation}\label{eq1}
(\rho_1\rho_2(\rho_1\rho_0)^2)^b(\rho_2\rho_1(\rho_0\rho_1)^2)^c = 1
\end{equation}
to the universal Coxeter group $[6,3]$ for $b=0$, $c=0$ or $b=c$. The resulting regular map is denoted by $\{6,3\}_{(b,c)}$.

Doubling the fundamental region of the C-group or, in other words, looking at the subgroup $H$ generated by $\rho_1^{\rho_0}, \rho_1, \rho_2$, we get the following non-string Coxeter diagram and we shall also say that $H$ is a non-string C-group and denote it by $[(3,3,3)_{(b,c)}]$.
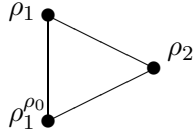
\begin{center}
\begin{picture}(40,55)
\put(0,45){\circle*{5}}
\put(40,25){\circle*{5}}
\put(0,5){\circle*{5}}
\put(0,5){\line(2,1){40}}
\put(0,45){\line(2,-1){40}}
\put(0,5){\line(0,1){40}}
\put(-15,5){$\rho_1^{\rho_0}$}
\put(-15,45){$\rho_1$}
\put(45,30){$\rho_2$}
\end{picture}
\end{center}
Taking a quotient of $H$ by adding relation (\ref{eq1}), we get a finite incidence geometry that can be seen as a regular hypermap $(3,3,3)_{(b,c)}$ on the torus. 
Observe also that the hypermaps $(3,3,3)_{(b,c)}$ are hypertopes whenever $(b,c) \neq (1,1)$. 
Indeed, the subgroup $H$ may be also seen as a group acting on the the tessellation of the plane by hexagons where hexagons are split into three families giving the three types, no two hexagons of the same type having a common edge and two hexagons being incident provided they have a common edge. This geometry is obviously residually connected and any $(3,3,3)_{(b,c)}$ is obtained by quotienting this geometry, relation~(\ref{eq1}) giving a partition of the hexagons in equivalence classes. 
Hence residually connectedness remains true in the quotiented geometry.
When $(b,c) = (1,1)$, there are only 3 elements of each type, hence thinness is lost as we showed in Example~\ref{ex44}. When $(b,c) \neq (1,1)$, there are enough elements of each type to have thinness.
As explained in~\cite{BN}, a hypermap $(3,3,3)_{(b,c)}$ with $(b,c)\neq (1,1)$ is a regular hypertope if and only if $bc(b-c) = 0$. Otherwise it is chiral.
We call such hypertopes \emph{toroidal hypertopes} of rank 3.


\section{C$^+$-groups}\label{section6}

In constructing regular hypertopes from groups, we restricted ourselves to groups generated by involutions.
We now consider a class of groups that are not necessarily generated by involutions, from which we will be able to construct highly symmetric hypertopes. These hypertopes may or may not be regular. In the latter case, they will be chiral. 

Consider a pair $(G^+,R)$ with $G^+$ being a group and $R:=\{\alpha_1, \ldots, \alpha_{r-1}\}$ a set of generators of $G^+$.
Define $\alpha_0:=1_{G^+}$
 and $\alpha_{ij} := \alpha_i^{-1}\alpha_j$ for all $0\leq i,j \leq r-1$.
Let $G^+_I := \langle \alpha_{ij} \mid i,j \in I\rangle$ for $I\subseteq \{0, \ldots, r-1\}$.

If the pair $(G^+,R)$ satisfies condition (\ref{IC+}) below called the {\em intersection property} IP$^+$ (obtained in analogy with the intersection property~(\ref{IC}) of C-groups  keeping only those equalities that involve subsets $I$ and $J$ of cardinality at least two), we say that $(G^+,R)$ is a {\em $C^+$-group}.
\begin{equation}\label{IC+}
\forall I, J \subseteq \{0, \ldots, r-1\}, with \;|I|, |J| \geq 2,
G^+_I \cap G^+_J = G^+_{I\cap J}.
\end{equation}
It follows immediately from the intersection property IP$^+$, that $R$ is an independent generating set for $G^+$, that means that $\alpha_i \not\in \langle \alpha_j : j \neq i\rangle$.

Examples of C$^+$-groups may be constructed from C-groups as follows.
Given a C-group $(G,S)$ with  $S :=\{\rho_0, \ldots, \rho_{r-1}\}$, we define the {\it rotation subgroup} $(G^+,R)$
where $R := \{ \alpha_j := \rho_0\rho_j : j \in \{1, \ldots, r-1\}\}$ and $G^+ := \langle R \rangle$. Let $\alpha_0:=1_{G^+}$. Obviously, $\alpha_{ij} := \rho_i\rho_j = \alpha_i^{-1}\alpha_j \in G^+$ for any choice of $i, j\in \{0, \ldots, r-1\}$.
The subgroup $G^+$ is of index 1 or 2 in $G$.

\begin{proposition}\label{rigs}
Let $(G,S)$ be a C-group and $R$ be defined as above. The set $R$ is an independent generating set for $G^+$. 
\end{proposition}
\begin{proof}
Suppose $R$ is not an independent generating set. Then there exists $i$ such that $\alpha_i \in \langle \alpha_j | j \in \{1,\ldots, r-1\} \setminus\{i\}\rangle$.
Hence $\rho_0\rho_i \in \langle \rho_0\rho_j | j \in \{1,\ldots, r-1\} \setminus\{i\}\rangle$. But $(G,S)$ is a C-group. Thus $\langle \rho_0,\rho_i\rangle \cap \langle \rho_0,\rho_j | j \in \{1,\ldots, r-1\} \setminus\{i\}\rangle$ must be equal to $\langle \rho_0 \rangle$. Therefore $\rho_0 = \rho_i$ and $(G,S)$ is not a C-group.
\end{proof}

\begin{proposition}\label{C+}
Let $(G,S)$ be a C-group and $(G^+,R)$ its rotation subgroup as defined above.
If $G^+$ is of index 2 in $G$ then $(G^+,R)$ is a C$^+$-group.
\end{proposition}
\begin{proof}
By Proposition~\ref{rigs}, $R$ is an independent generating set.
Moreover, as $G^+$ is of index 2 in $G$, $\rho_i\not\in G^+$ for every $i=0,\ldots ,r-1$.
Suppose $(G^+,R)$ is not a $C^+$-group. Then there exist $I,J \subseteq \{0, \ldots, r-1\}$ with $G^+_I \cap G^+_J > G^+_{I\cap J}$.
We have $G_{I\cap J} = G_I\cap G_J \geq G^+_I \cap G^+_J > G^+_{I\cap J}$. Moreover, the index of $G^+_{I\cap J}$ in $G_{I\cap J}$ is at most 2, hence it is 2 for otherwise $G^+_I \cap G^+_J = G^+_{I\cap J}$.
This imply that $G^+_I\cap G^+_J = G_{I \cap J}$. Therefore $I\cap J$ must be empty for otherwise there exists a $\rho_i\in G^+$, a contradiction.
Now, if $I\cap J$ is empty, we have $\{1_G\} = G_{I\cap J} = G_I\cap G_J \geq G^+_I \cap G^+_J > G^+_{I\cap J}$ which is also impossible.
Therefore $(G^+, R)$ is a $C^+$-group.

\end{proof}

The next example shows that it is possible to construct examples of C-groups $(G,S)$ where $(G^+,R)$ does not satisfy $IP^+$ when $G^+ = G$.
\begin{example}
Let

$\rho_0 := (3, 6)(4, 7)(5, 9)(8, 10)(11, 12),$

$\rho_1 := (2, 4)(3, 6)(5, 10)(8, 12)(9, 11),$

$\rho_2 := (1, 2)(3, 5)(4, 7)(6, 9)(8, 11)(10, 12),$ and 
    
    $\rho_3:= (1, 3)(2, 5)(4, 8)(7, 11)(10, 12)$.
    
    Let $S:= \{\rho_0,\rho_1,\rho_2,\rho_3\}$ and $G := \langle S\rangle$.
    It can be checked by hand or using {\sc Magma} that $(G,S)$ is a C-group but that $(G^+,R)$ is not a $C^+$-group as can be found by taking $I:=\{0,1,2\}$ and $J :=\{1,2,3\}$ to check that (\ref{IC+}) is not satisfied.
\end{example}
Let $(G^+,R)$ be a C$^+$-group.
It is convenient to represent $(G^+, R)$ by the following complete  graph with $r$ vertices which we will call the {\it $B$-diagram} of $(G^+,R)$ and denote by $\mathcal{B}(G^+,R)$. The vertex set of $\mathcal{B}$ is the set $\{\alpha_0, \ldots, \alpha_{r-1}\}$. Since $\alpha_i\alpha_j^{-1} = (\alpha_i^{-1}\alpha_j)^{\rho_0}$, the edges $\{\alpha_i,\alpha_j\}$ of this graph are labelled by $o(\alpha_i^{-1}\alpha_j) = o(\alpha_j^{-1}\alpha_i)= o(\alpha_i\alpha_j^{-1})$ . We take the convention of dropping an edge if its label is 2 and of not writing the label if it is 3.
Vertices of $\mathcal{B}$ are represented by small circles in order to distinguish from the vertices of a Coxeter diagram, which represent involutions.
Observe that a C-group $(G,S)$ and its corresponding C$^+$-group $(G^+,R)$ will have isomorphic diagrams. The main difference is that the vertex set of the Coxeter diagram of $(G,S)$ is $S$ while the vertex set of the $B$-diagram of a $(G^+,R)$ is $R\cup \{1_{G^+}\}$.

For instance, the automorphism group of a chiral 4-polytope of type $\{6,3,3\}$ with toroidal facets has the following $B$-diagram.
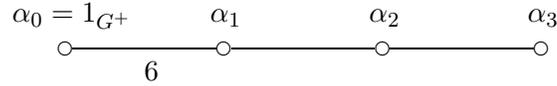
\begin{center}
\begin{picture}(180,40)
\put(0,15){\circle{5}}
\put(60,15){\circle{5}}
\put(120,15){\circle{5}}
\put(180,15){\circle{5}}
\put(30,3){6}
\put(3,15){\line(1,0){54}}
\put(63,15){\line(1,0){54}}
\put(123,15){\line(1,0){54}}
\put(-20,25){$\alpha_0=1_{G^+}$}
\put(55,25){$\alpha_1$}
\put(115,25){$\alpha_2$}
\put(175,25){$\alpha_3$}
\end{picture}
\end{center}


In~\cite{SW1991}, the set of generators of $G^+$ is usually denoted by $\sigma_1, \ldots, \sigma_{r-1}$.
We note that in the example above, $\alpha_1 = \sigma_1$, $\alpha_2 = \sigma_1\sigma_2$ and $\alpha_3 = \sigma_1\sigma_2\sigma_3$, and more generally, given an abstract chiral polytope of Sch\"afli type $\{p_1, p_2,\ldots, p_{r-1}\}$ with generators $\sigma_1, \ldots, \sigma_{r-1}$, we have $\alpha_1 = \sigma_1$, $\alpha_2 = \sigma_1\sigma_2$ and $\alpha_i = \sigma_1\sigma_2\sigma_3\ldots \sigma_i$ for $2\leq i\leq r-1$.
Given an abstract chiral polytope and a set of generators $\sigma_1, \ldots, \sigma_{r-1}$ of its automorphism group $G^+$ as in~\cite{SW1991}, there is no automorphism $g$ of $G^+$ such that $g(\sigma_1) = \sigma_1^{-1}$, $g(\sigma_2) = \sigma_1^2\sigma_2$ and $g(\sigma_i) = \sigma_i$ for $3\leq i \leq r-1$. 
If such an automorphism exists, the polytope is regular. In this case the group $G^+$ is of index one or two in a C-group $(G,S)$ with $S = \{\rho_0, \ldots, \rho_{r-1}\}$ where $\rho_0 :=g$ and $\rho_i=g\alpha_i$ for $i=1,\ldots, r-1$. 
In terms of the generators $\{\alpha_i : i\in \{1,\ldots, r-1\}\}$, this condition is equivalent to having no $g\in Aut(G^+)$ such that $g(\alpha_i) = \alpha_i^{-1}$ for all $1\leq i\leq r-1$. The generators $\sigma_i$'s are a natural choice in the case of geometries with a linear diagram. We adopt a different set of generators, needed for the cases where the diagram is not linear. Our generators correspond in rank three to the ones usually chosen in the literature on maps and hypermaps.

\section{Chiral hypertopes as coset geometries}\label{section7}

Given a chiral hypertope $\Gamma(X,*,t,I)$ (with $I:=\{0, \ldots, r-1\}$) and its automorphism group $G^+:=Aut_I(\Gamma)$, pick a chamber $C$.
For any pair $i\neq j \in I$, there exists an automorphism $\alpha_{ij}\in G^+$ that maps $C$ to $(C^i)^j$.
Also, $C\alpha_{ii}=(C^i)^i=C$ and $\alpha_{ij}^{-1}=\alpha_{ji}$.
We define 
\[\alpha_i := \alpha_{0i}\;(i=1\ldots, r-1)\]
and call them the {\it distinguished generators} of $G^+$ with respect to $C$.

Arguments similar to those used in the proof of Theorem~\ref{cgroup} permit to show that the pair $(G^+,R)$ is a C$^+$-group, that is, the distinguished generators of $G^+$ satisfy the intersection property IP$^+$ and the relations implicit in some $B$-diagram with vertex set $R\cup\{\alpha_0\}$ where $\alpha_0:= 1_{G^+}$.
The following theorem is the chiral equivalent of Theorem~\ref{cgroup}.

\begin{theorem}\label{cplusgroup}
Let $I:=\{0, \ldots, r-1\}$ and let $\Gamma$ be a chiral hypertope of rank $r$. Let $C$ be a chamber of $\Gamma$.
The pair $(G^+,R)$ where $G^+=Aut_I(\Gamma)$ and $R$ is the set of distinguished generators of $G^+$ with respect to $C$ is a $C^+$-group.
\end{theorem}
\begin{proof}
We assume that $r\geq 3$ as no chiral hypertope of rank $\leq 2$ exist.
Let us again denote by $G^+_{\overline{K}}$ the subgroup $\langle \alpha_{ij} | i, j \in I\setminus K\rangle$.
If $|K|\geq r-1$, then $G^+_{\overline{K}} = \{1\}$.
Let $G_i$ be the stabiliser in $G^+$ of the element of type $i$ in $C$ for $i=0, \ldots, r-1$.

If $(G^+,R)$ does not satisfy ($\ref{IC+}$), then there is a pair of subgroups, $G^+_{\overline{K}}$ and $G^+_{\overline{J}}$ with $K,J \subseteq I$, both of size at most $r-2$, such that $G^+_{\overline{K}} \cap G^+_{\overline{J}} \neq G^+_{\overline{K \cup J}}$.
Hence $G^+_{\overline{K}} \cap G^+_{\overline{J}} > G^+_{\overline{K \cup J}}$. Take $g \in (G^+_{\overline{K}} \cap G^+_{\overline{J}}) \setminus G^+_{\overline{K \cup J}}$.
This $g$ fixes a flag of type $K \cup J$ in the base chamber $C =\{G_0, \ldots, G_{r-1}\}$.
But the action of $G^+_{\overline{K\cup J}}$ must be free on the residue $\Gamma_F$ of the flag $F:=\{G_i \mid i \in K \cup J\}$ for otherwise the action of $G^+$ is not free on $\Gamma$, a contradiction.
Any element of $G^+_{\overline{K\cup J}}$ will fix all elements of $F$.
If $|K\cup J| = r$, the element $g$ fixes a chamber, a contradiction with the fact that the action is free.
If $|K\cup J| = r-1$, there are exactly two chambers containing the flag $F$. Since the action is chiral, $g$ must also fix these chambers, and we conclude that the action is not free, a contradiction.
Finally, suppose that $|K \cup J| < r-1$.
Since $G^+_{\overline{K\cup J}}$ has at most two orbits on the flags of $\Gamma_F$, there must exist an element $h \in G^+_{\overline{K\cup J}}$ that sends the flag $\{G_k \mid k \in I\setminus (K\cup J)\}$ onto $\{G_k g \mid k \in I\setminus (K\cup J)\}$ or onto one of its adjacent flags.
In the first case, $gh^{-1}\not = 1_G$ fixes the base chamber $\{G_i \mid i \in I\}$, the action of $G$ on the chambers of $\Gamma$ is not free, a contradiction.
In the second case, $gh^{-1}\not =1_G$ maps the base chamber onto one of its adjacent chambers, contradicting chirality of the hypertope.
\end{proof}

\begin{corollary}
The set $R$ of Theorem~\ref{cplusgroup} is an independent generating set for $G^+$. 
\end{corollary}
\begin{proof}
Assume $R$ is not an independent generating set. Then there exists $\alpha_i$ in $R$ such that $\alpha_i \in \langle \alpha_j | j \in \{1,\ldots,r-1\} \setminus \{i\}\rangle$.
But then, the groups $G^+_I$ and $G^+_J$ with $I:=\{0,i\}$ and $J:=\{0,\ldots ,r-1\}\setminus \{i\}$ contradict (\ref{IC+}).
\end{proof}

The notion of chirality in incidence geometries was well explored in the case when the diagram of the geometry is linear, that is when the induced chiral geometry is an abstract chiral polytope~\cite{SW1991}.
The automorphism  groups of chiral polytopes are characterized as groups with specific generators $\sigma_1, \ldots, \sigma_{r-1}$ such that $\sigma_i\sigma_{i+1}\ldots \sigma_j$ is of order 2 for all $1\leq i < j \leq r-1$.
Examples can be found where the $\sigma_i$'s are not independent when $r\geq 4$. 
For instance, all chiral polytopes with automorphism group $S_5$ given in~\cite{HHL} have their $\sigma_i$'s not independent.

\section{Constructions and examples}\label{section8} 
Just as regular polytopes can be constructed inductively from regular polytopes of lower rank, similar constructions can be applied to hypertopes. However, such constructions of chiral polytopes are not possible as the $(n-2)$-faces of a chiral polytope of rank $n$ are necessarily regular (see~\cite[Proposition 9]{SW1991}).
Although the result for polytopes cannot be extended to thin geometries with a nonlinear diagram (as these geometries are not necessary posets),
the following result  imposes similar restriction on inductive constructions of chiral hypertopes.

Let $\Gamma(X,*,t,I)$ be an incidence geometry. For $J\subseteq I$, we define the {\em $J$-truncation} of $\Gamma$ as the incidence geometry $^J\Gamma(^JX, ^J*, ^Jt, J)$ where $^JX = t^{-1}(J)$, $^J*$ is the restriction of $*$ to $^JX \times ^JX$ and $^Jt$ is the restriction of $t$ to $^JX$. If $\Gamma(G; (G_i)_{i\in I})$ is a coset geometry and $J\subseteq I$, the $J$-truncation of $\Gamma(G; (G_i)_{i\in I})$ is the coset geometry $\Gamma(G; (G_j)_{j\in J})$.

\begin{theorem}\label{truncations}  
Let $\Gamma(X,*,t,I)$ be a chiral hypertope of rank $|I| \geq 3$. Then, any rank $|I|-2$ truncation of $\Gamma$ is a flag-transitive geometry.
\end{theorem}
\begin{proof}
Let $J\subset I$ with $|J| = |I|-2$ and let $k, l \in I \setminus J$.
Take two chambers $^JC_1$ and $^JC_2$ of $^J\Gamma$.
These two chambers are flags of $\Gamma$ and as $\Gamma$ is a geometry, they can be extended to chambers $C_1 :=$ $^JC_1 \cup \{x_k,x_l\}$ and $C_2 :=$ $^JC_2 \cup \{y_k,y_l\}$ of $\Gamma$.
If the chambers $C_1$ and $C_2$ are in the same orbit under the action of $Aut_I(\Gamma)$, then $^JC_1$ and $^JC_2$ are in the same orbit under the action of $Aut_I(^J\Gamma)$.
Suppose $C_1$ and $C_2$ are not in the same orbit. Hence, $C_2$ can be mapped to an adjacent chamber of $C_1$, namely $C'_1:=$ $^JC_1 \cup \{x'_k, x_l\}$ where $x'_k \neq x_k$. Hence there exists $g\in Aut_I(\Gamma)$ such that $g(C_2) = C'_1$. In particular, $g(^JC_2) =$ $^JC_1$, and thus $^JC_1$ and $^JC_2$ are in the same orbit under the action of $Aut_I(^J\Gamma)$. Therefore, $^J\Gamma$ is flag-transitive.
\end{proof}

We note that the rank $|I|-2$ truncations of a chiral hypertope need not be (regular) hypertopes as in most cases thinness is lost.
Nevertheless, we may be able to construct hypertopes of rank $r$ from flag-transitive geometries of rank $r-2$.



In Section~\ref{section4}, 
we saw that to a regular hypertope and one of its chambers, we can associate a $C$-group (Theorem~\ref{cgroup}). 
We constructed regular hypertopes from C-groups under certain conditions.
We showed in example~\ref{ex4.5} that a given $C$-group $(G,S)$ does not necessarily give a coset geometry that is a regular hypertope but in Theorem~\ref{theorem46}, we showed that if $G$ is flag-transitive on the coset geometry, then that geometry is a hypertope.

As we shall see now, we can construct chiral and regular hypertopes from some C$^+$-groups $(G^+,R)$
where $R$ is a set of independent generators.


We start by showing how to construct a coset geometry from a group and an independent generating set of this group.

\begin{construction}\label{hyper}
Let $I=\{1,\ldots, r-1\}$,  $G^+$ be a group and $R$ be an independent generating set of $G^+$.
Define $G^+_i := \langle \alpha_j | j \neq i \rangle$ for $i=1, \ldots, r-1$ and $G^+_0 := \langle \alpha_1^{-1}\alpha_j | j \geq 2 \rangle$.
The coset geometry $\Gamma(G^+,R) := \Gamma(G^+; (G^+_i)_{i\in \{0,\ldots,r-1\}})$ constructed using Tits' algorithm (see Proposition~\ref{tits}) is the geometry associated to the pair $(G^+,R)$. 
\end{construction}

The coset geometry $\Gamma(G^+,R)$ gives an incidence system using Tits algorithm. 
If this incidence system is a chiral hypertope, then $(G^+,R)$ is necessarily a $C^+$-group as we showed in Theorem~\ref{cplusgroup}, so in order to construct chiral hypertopes from coset geometries, it is necessary to start with $C^+$-groups.

If one looks at the rotation subgroup of the group $A_6$ of Example~\ref{ex4.5}, it is clear that a C$^+$-group does not necessarily give a coset geometry $\Gamma$ that is thin and strongly chamber connected and hence does not give automatically a hypertope.

\begin{theorem}\label{elisa}
Let $(G^+,R)$ be a $C^+$-group.
Let $\Gamma := \Gamma(G^+,R)$ be the coset geometry associated to $(G^+,R)$ using Construction~\ref{hyper}.
If $\Gamma$ is a hypertope and $G^+$ has two orbits on the set of chambers of $\Gamma$, then 
$\Gamma$ is chiral if and only if there is no automorphism of $G^+$ that inverts all the elements of $R$.
Otherwise, there exists an automorphism $\sigma\in Aut(G^+)$ that inverts all the elements of $R$ and the group $G^+$ extended by $\sigma$ is regular on $\Gamma$.
\end{theorem}
\begin{proof}
As $\Gamma$ is a hypertope, it is thin and residually connected, hence strongly chamber connected by Proposition~\ref{rcsc}.
This implies that $G^+$ acts freely on the chambers of $\Gamma$.

Let us show that the two orbits of chambers under the action of $G^+$ are obtained from the base chamber $C:=\{G_0^+, \ldots, G_{r-1}^+\}$ and its $i$-adjacent chamber.
Given the base chamber $C$, its $i$-adjacent chamber is $(C\setminus \{G_i^+\}) \cup \{G_i^+g\}$ where $g := \alpha_1^{-1}$ if $i=0$ and $g :=\alpha_i$ otherwise.
Indeed, if $i=0$, $G_0^+\alpha_1^{-1}$ contains $\alpha_j^{-1}$ for all $j=1, \ldots, r-1$ and thus $G_0^+\alpha_1^{-1}$ has at least one element in common with $G_i^+$ for every $i=1, \ldots, r-1$.
Hence $\{G_0^+\alpha_1^{-1},G_1^+, \ldots, G_{r-1}^+\}$ is the 0-adjacent chamber to $C$.
Similarly $G_i^+\alpha_i$ contains $\alpha_j^{-1}\alpha_i$ for all $j=1, \ldots, r-1$ and thus $G_i^+\alpha_i$ has at least one element in common with $G_j^+$ for every $j=0, \ldots, r-1$, with $j\neq i$. Also $G_0^+$ contains $\alpha_j^{-1}\alpha_i$ for every $j\neq i$ and hence $G_0^+$ is incident to $G_i^+\alpha_i$. Thus the set $(C\setminus \{G_i^+\}) \cup \{G_i^+\alpha_i\}$ is the $i$-adjacent chamber to $C$.

As $(G^+,R)$ is a $C^+$-group, there is no element of $G^+$ that maps $C$ to any of its $i$-adjacent chambers. This is due to the fact that, in order to map $C$ to its $i$-adjacent chamber, we need to fix all but one of the maximal parabolic subgroups, and the intersection of all but one maximal parabolic subgroups is the identity.
Thus for every $i=0, \ldots, r-1$, $C$ and its $i$-adjacent chamber are not in the same orbit under the action of $G^+$.
Hence any chamber in the orbit of $C$ under the action of $G^+$ will not be in the same orbit with its $i$-adjacent chamber under the action of $G^+$.

If there exists an element $\sigma\in Aut(G^+)$ that inverts every generator, then $\sigma$ obviously maps $C$ to its $0$-adjacent chamber and hence fuses the two chamber orbits. The group $G^+$ extended with $\sigma$ will therefore act regularly on $\Gamma$.
On the other hand, if no such element exists, $\Gamma$ has two orbits on its set of chambers, and two adjacent chambers are always in distinct orbits, hence $\Gamma$ is chiral.
\end{proof}
We conclude this section with a concrete construction leading to rank 4 regular and chiral hypertopes.
Whenever the group $[(3,3,3)_{(b,c)}]$ defined in Section~\ref{section5} is the automorphism group of a rank 3 hypertope, we extend it by an involution $\rho_3$ such that $\rho_2\rho_3$ is of order $p$, and in addition $\rho_3$ commutes with $\rho_0$ and $\rho_1$. We get a rank four C-group with the following diagram.
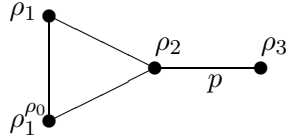
\begin{center}
\begin{picture}(80,55)
\put(0,47){\circle*{5}}
\put(40,27){\circle*{5}}
\put(80,27){\circle*{5}}
\put(0,7){\circle*{5}}
\put(0,7){\line(2,1){40}}
\put(0,47){\line(2,-1){40}}
\put(0,7){\line(0,1){40}}
\put(40,27){\line(1,0){40}}
\put(-15,5){$\rho_1^{\rho_0}$}
\put(-15,47){$\rho_1$}
\put(40,34){$\rho_2$}
\put(80,34){$\rho_3$}
\put(60,20){$p$}
\end{picture}
\end{center}

When $p=3$, $4$, $5$ and $6$, this group is a subgroup of index 2 in the Coxeter group $[6,3,p]$ (generated by $\rho_0$, $\rho_1$, $\rho_2$ and $\rho_3$), the symmetry group of a regular tessellation of hyperbolic $3$-space by horospherical cells $\{6,3\}$ (\cite{Coxeter1950}, ~\cite{Coxeter1954}). 

Adding the relations~(\ref{eq1}) to $[6,3,p]$, we obtain the universal regular polytopes $\{\{6,3\}_{(b,c)},\{3,p\}\}$ with $p=3,4,5$.
Finite regular polytopes in this family have been classified in~\cite{ARP}.
They are included in Table~\ref{pol63p} where we also list the chiral polytopes of that type.
We conjecture that this list is also complete based on~\cite{Col} and computations done using {\sc Magma}. 
With the exception of $\{\{6,3\}_{(1,1)},\{3,p\}\}$, as explained by Example~\ref{ex44}, each of these regular and chiral polytopes is a double cover of the corresponding hypertope $\{(3,3,3)_{(b,c)},p\}$ (whose  residues of $\rho_3$ are toroidal hypermaps, while the residues of $\rho_1$ and $\rho_1^{\rho_0}$ are toroidal maps). Regular hypertopes of this type are thus also classified and our conjecture extends to chiral hypertopes as well. 
Table~\ref{333p} lists the hypertopes obtained using this construction.
More precisely, they are obtained by using the following presentation for the rotational subgroup with $p$ and $\mathbf{s} = (a,b)$ as parameters.
\[G^+(p,a,b):= \langle x,y,z | x^3,
y^3,
z^p,
(x^{-1}z)^2,
(y^{-1}z)^2,
(x^{-1}y)^3,
(xy^{-1}x)^a(xy)^b
\rangle
\]
The case where $p=6$ is considerably more complicated and will be dealt with in another paper~\cite{FLW}.

\begin{table}
\begin{tabular}{|c|c|c|c|c|c|c|c|}
\hline
$p$&$\mathbf s$&$g$&Group&
Chiral/Regular\\
\hline
3 & $( 2 , 0 )$& 240 & $S_5\times C_2$&
regular \\
 & $( 3 , 0 )$& 1296 & $[1$ $1$ $2]^3 \rtimes C_2$&
 regular \\
 & $( 4 , 0 )$& 15360 & $[1$ $1$ $2]^4 \rtimes C_2$&
 regular \\
 & $( 1 , 2 )$& 336 & $PGL_2(7)$&
 chiral \\
 & $( 1 , 3 )$& 2184 & $PSL_2(13)\times C_2$&
 chiral \\
 & $( 1 , 4 )$& 8064 & $SL_2(7)\rtimes A_4\rtimes C_2$&
  chiral \\
 & $( 2 , 2 )$& 2880 & $S_5\times S_4$&
 regular \\
 & $( 2 , 3 )$ & 6840 & $PGL_2(19)$&
 chiral \\
\hline
4 & $( 1 , 1 )$ & 288 & $S_3 \rtimes [3,4]$ &
regular \\
 & $( 2 , 0 )$ & 768 & $[3, 3, 4] \times C_2$&
 regular \\
 & $( 1 , 2 )$ & 2016 & $PGL_2(7) \times S_3$&
  chiral \\
\hline
5 & $( 2 , 0 )$&  28800 & $[3, 3, 5]\times C_2$&
regular \\
\hline
\end{tabular}
\caption{Finite polytopes of type $\{\{6,3\}_\mathbf{s},\{3,p\}\}$ with $p\in\{3,4,5\}$ (having $g$ flags)}\label{pol63p}
\end{table}

\begin{table}
\begin{tabular}{|c|c|c|c|c|}
\hline
$p$&$\mathbf s$&$g$&Group&
Chiral/Regular\\     
\hline
3 & $( 2 , 0 )$& 120 & $S_5$&
regular \\
 & $( 3 , 0 )$ & 648 & $[1$ $1$ $2]^3$ &
  regular \\
 & $( 4 , 0 )$ & 7680 & $[1$ $1$ $2]^4$&
 regular \\
 & $( 1 , 2 )$ & 168 & $PSL_2(7)$ &
 chiral \\
 & $( 1 , 3 )$ & 1092 & $PSL_2(13)$ &
  chiral \\
 & $( 1 , 4 )$ & 4032 & $SL_2(7)\rtimes A_4$&
 chiral \\
 & $( 2 , 2 )$ & 1440 & $A_5\times S_4$&
 regular \\
 & $( 2 , 3 )$ & 3420 & $PSL_2(19)$&
 chiral \\
\hline
4  & $( 1 , 2 )$ & 1008 & $PSL_2(7)\times S_3$&
chiral \\
 & $( 2 , 0 )$ & 384 & $[3,3,4]$&
 regular \\
\hline
5 & $( 2 , 0 )$ & 14400 & $[3,3,5]$ &
regular \\
\hline
\end{tabular}
\caption{Finite hypertopes of type $\{(3,3,3)_\mathbf{s},p\}$ with $p\in\{3,4,5\}$ (having $g$ flags)}\label{333p}
\end{table}

\section{Open problems and acknowledgements}\label{section9}
We conclude this paper with a series of open problems that we think are interesting to investigate in future work.

\begin{problem}
What is a minimal set of conditions for the IP$^+$ condition?
\end{problem}
The corresponding problem in polytopes has been solved in~\cite{Conder2013} by Conder and Oliveros.

\begin{problem}
Classify all finite locally toroidal incidence geometries of type $\{(3,3,3);p\}$.
\end{problem}
\begin{problem}
Find an example of a C-group that gives a thin, residually connected geometry of rank $\geq 4$ that is not flag-transitive.
\end{problem}
\begin{problem}
Find an example of a C-group of rank 3 that gives a geometry that is not thin, not residually connected and not flag-transitive.
\end{problem}

This research was supported by a Marsden grant (UOA1218) of the Royal Society of New Zealand, by NSERC and 
by the Portuguese Foundation for Science and Technology (FCT-Fundação para a Ciência e a Tecnologia), through CIDMA - Center for Research and Development in Mathematics and Applications, within project UID/MAT/04106/2013.

The authors would like to warmly thank an anonymous referee whose numerous comments helped improve greatly this article.

\bibliographystyle{plain}

\begin{thebibliography}{10}

\bibitem{Aschbacher83}
Michael Aschbacher.
\newblock Flag structures on {T}its geometries.
\newblock {\em Geom. Dedicata}, 14(1):21--32, 1983.

\bibitem{magma}
Wieb Bosma, John J. Cannon, and Catherine Playoust, \emph{The {M}agma {A}lgebra {S}ystem {I}:
  the user language}, J. Symbolic Comput. 3/4:235--265, 1997.

\bibitem{BN}
Antonio Breda D'Azevedo and Roman Nedela.
\newblock Chiral hypermaps of small genus.
\newblock {\em Beitr\"age Algebra Geom.}, 44(1):127--143, 2003.

\bibitem{Buek79}
Francis Buekenhout.
\newblock Diagrams for geometries and groups.
\newblock {\em J. Combin. Theory Ser. A}, 27(2):121--151, 1979.

\bibitem{Buek83}
Francis Buekenhout.
\newblock {$(g,\,d^{\ast} ,\,d)$}-gons.
\newblock In {\em Finite geometries ({P}ullman, {W}ash., 1981)}, volume~82 of
  {\em Lecture Notes in Pure and Appl. Math.}, pages 93--111. Dekker, New York,
  1983.

\bibitem{Buek95}
Francis Buekenhout, editor.
\newblock {\em Handbook of Incidence Geometry: Buildings and Foundations}.
\newblock Elsevier Science, North-Holland, 1995.

\bibitem{BuekCohen}
Francis Buekenhout and Arjeh~M. Cohen.
\newblock {\em Diagram Geometry. Related to classical groups and buildings}.
\newblock Ergebnisse der Mathematik und ihrer Grenzgebiete. 3. Folge. A Series of Modern Surveys in Mathematics [Results in Mathematics and Related Areas. 3rd Series. A Series of Modern Surveys in Mathematics], 57. Springer, Heidelberg, 2013. xiv+592 pp.

\bibitem{BDL96b}
Francis Buekenhout, Michel Dehon, and Dimitri Leemans.
\newblock {\em An {A}tlas of residually weakly primitive geometries for small
  groups}.
\newblock M\'em. Cl. Sci., Coll. 8, Ser. 3, Tome XIV. Acad. Roy. Belgique,
  1999.

\bibitem{BH91}
Francis Buekenhout and Michel Hermand.
\newblock On flag-transitive geometries and groups.
\newblock {\em Travaux de Math\'ematiques de l'Universit\'e Libre de
  Bruxelles}, 1:45--78, 1991.

\bibitem{Col}
Charles J. Colbourn and Asia~Ivi{{\'c}} Weiss.
\newblock {\em A census of regular 3-polystroma arising from honeycombs}.
\newblock Discrete Math. 50:29--36, 1984.




\bibitem{Conder2013}
Marston Conder and Deborah Oliveros.
\newblock {\em The intersection condition for regular polytopes}.
\newblock J. Combin. Theory Ser. A 120(6);1291--1304, 2013.

\bibitem{CJL2013}
Thomas Connor, Sebastian Jambor, and Dimitri Leemans.
\newblock C-groups of $PSL(2,q)$ and $PGL(2,q)$.
\newblock J. Algebra, to appear.

\bibitem{CoxeterMoserED4}
Harold~S.~M. Coxeter and William~O.~J. Moser.
\newblock {\em Generators and relations for discrete groups}, volume~14 of {\em
  Ergebnisse der Mathematik und ihrer Grenzgebiete [Results in Mathematics and
  Related Areas]}.
\newblock Springer-Verlag, Berlin, fourth edition, 1980.

\bibitem{Coxeter1950}
Harold~S.~M. Coxeter and G. J. Whitraw.
\newblock World-structure and non-Euclidean honeycombs.
\newblock {\em Proc. Roy. Soc. Lond. Ser. A 201}; 417--437, 1950.

\bibitem{Coxeter1954}
Harold~S.~M. Coxeter.
\newblock Regular honeycombs in hyperbolic space.
\newblock {\em Proc. Internat. Congress Math. Amsterdam (1954), Vol. 3}; 155--169.
\newblock North-Holland, Amsterdam, 1956.

\bibitem{Dehon94}
Michel Dehon.
\newblock Classifying geometries with \textsc{Cayley}.
\newblock {\em J. Symbolic Comput.}, 17(3):259--276, 1994.

\bibitem{FL1}
Maria Elisa Fernandes and Dimitri Leemans.
\newblock Polytopes of high rank for the symmetric groups.
\newblock {\em Adv. Math.}, 228(6):3207--3222, 2011.

\bibitem{FLW}
Maria Elisa Fernandes, Dimitri Leemans and Asia Ivi\'c Weiss.
\newblock {\em Hexagonal extensions of toroidal hypermaps}.
\newblock Preprint, 19 pages, 2015.

\bibitem{FLW2}
Maria Elisa Fernandes, Dimitri Leemans and Asia Ivi\'c Weiss.
\newblock {\em Further extensions of toroidal hypermaps}.
\newblock In preparation.

\bibitem{HHL}
Michael I. Hartley, Isabel Hubard and Dimitri Leemans.
\newblock {\em Two atlases of abstract chiral polytopes for small groups}.
\newblock Ars Math. Contemp. 5:371--382, 2012.

\bibitem{HW2005}
Isabel Hubard and Asia~Ivi{{\'c}} Weiss.
\newblock {\em Self-duality of chiral polytopes}.
\newblock J. Combin. Theory Ser. A 111:128--138, 2005.

\bibitem{ARP}
Peter McMullen and Egon Schulte.
\newblock {\em Abstract regular polytopes}, volume~92 of {\em Encyclopedia of
  Mathematics and its Applications}.
\newblock Cambridge University Press, Cambridge, 2002.

\bibitem{Pasini94}
Antonio Pasini.
\newblock {\em Diagram Geometries}.
\newblock Oxford Science Publications, Oxford, 1994.

\bibitem{SW1991}
Egon Schulte and Asia~Ivi{{\'c}} Weiss.
\newblock Chiral polytopes.
\newblock In {\em Applied geometry and discrete mathematics}, volume~4 of {\em
  DIMACS Ser. Discrete Math. Theoret. Comput. Sci.}, pages 493--516. Amer.
  Math. Soc., Providence, RI, 1991.

\bibitem{Tits61}
Jacques Tits.
\newblock Sur les analogues alg\'ebriques des groupes semi-simples complexes.
\newblock Colloque d'alg\`ebre sup\'erieure, tenu \`a {B}ruxelles du 19
              au 22 d\'ecembre 1956, Centre Belge de Recherches Math\'ematiques, 261--289,
 \'Etablissements Ceuterick, Louvain; Librairie
              Gauthier-Villars, Paris, 1957.
 
\bibitem{Tits74}
Jacques Tits.
\newblock {\em Buildings of Spherical Type and Finite $BN$-Pairs}.
\newblock Springer Verlag Berlin Heidelberg New-York, 1974.

\end{thebibliography}

\end{document}